\documentclass[reqno]{amsart}
\usepackage{amsmath, amsthm, amssymb}

\title[Necessary and sufficient conditions \ldots]{Necessary and sufficient conditions for weak exponential instability of evolution operators}

\author{Nicolae Lupa}
\address{Author's Address. Department of Mathematics,
Faculty of Mathematics and Computer Science,
West University of Timi\c soara,
4 V.  P\^ arvan,  Blvd.,
300223 Timi\c soara, Romania}
\email{nlupa@math.uvt.ro}


\newtheoremstyle{theorem}
  {10pt}		  
  {10pt}  
  {\sl}  
  {\parindent}     
  {\bf}  
  {. }    
  { }    
  {}     
\theoremstyle{theorem}
\newtheorem{theorem}{Theorem}
\newtheorem{corollary}[theorem]{Corollary}
\newtheorem{proposition}[theorem]{Proposition}
\newtheorem{remark}[theorem]{Remark}
\newtheorem{example}[theorem]{Example}

\newtheoremstyle{defi}
  {10pt}		  
  {10pt}  
  {\rm}  
  {\parindent}     
  {\bf}  
  {. }    
  { }    
  {}     
\theoremstyle{defi}
\newtheorem{definition}[theorem]{Definition}

\begin{document}

\begin{abstract}
In this paper we give some necessary and sufficient conditions for weak exponential instability of evolution operators. Variants for the classical results due to Datko and Lyapunov are obtained.
\end{abstract}

\maketitle

\section{Introduction}

The stability theory has reached a considerable degree of maturity. In contrast with this, we do not yet know a coherent theory of unstable manifolds, even if in recent years new concepts of instability have been introduced and studied in \cite{Bi.2009},  \cite{Me.Po.2003}, \cite{Me.Sa.Sa.2001-1}, \cite{Mi.Ra.Sc.1998}, \cite{Pa.2006} and \cite{Me.St.2008-1}.

The aim of this paper is to give some necessary and sufficient conditions for the concept of weak exponential instability introduced in \cite{Lu.2009}.
We prove continuous and discrete variants for weak exponential instability of a well-known result from the stability theory due to Datko \cite{Da.1973} and generalized by Ichikawa \cite{Ic.1984}, van Neerven \cite{Ne.2002-1} and Pazy {\cite{Pa.1983}.

Our main objective is to propose a Lyapunov function to study the existence of weak exponential instability for evolution operators in Banach spaces.

The theory of Lyapunov functions has many application. For example, in \cite{Re.Ke.2010} the authors construct a type of Lyapunov function for a variety of SIR and SIRS models to obtain the stability of the endemic equilibrium states of epidemic systems and, in  \cite{Ra.2010} Ratchagit use a Lyapunov functional approach to establish exponential stability of linear systems with time varying delay.

The classical  theorem of Lyapunov states that if
A is a $n \times n$  complex matrix then A has all its characteristic roots lying in the half plane $Rez<0$ if and only if  the matrix equation $A^{*}B +BA= -I$  has a unique solution which is a positive definite Hermitian matrix.

This result was extended by Datko  \cite{Da.1970} for the case of $C_0$-semigroups in the following sense:

\begin{theorem}
A $C_0$-semigroup on a complex Hilbert space $X$ is uniformly exponentially stable if and only if there is an operator $B\in \mathcal{B}(X)$ with $B^{*}=B$ and $B\geq 0$ such that
\begin{equation}\label{eq.Ly-1}
<Ax,Bx>+<Bx,Ax>=-\parallel x\parallel^{2},\, \forall\,x\in D(A),
\end{equation}
where $A$ denotes the infinitesimal generator of the $C_0$-semigroup.
\end{theorem}

Condition (\ref{eq.Ly-1}) is equivalent with
\begin{equation}\label{eq.Ly-2}
A^{*}Q+QA=-I \text{ on } D(A),
\end{equation}
where $Q$ is a positive and bounded operator satisfying $QD(A)\subset D(A^{*})$.

Other Lyapunov type characterizations were obtained in \cite{Ba.Va.2009-2},  \cite{Gu.Zw.2006},  \cite{Po.Pr.Pr.2005} and \cite{Pr.Ic.1987}. We remark that these results hold in finite or in infinite dimensional Hilbert spaces. In this paper we extend this results in two direction. First, we consider the case of weak exponential instability and second, our approach is true even in Banach spaces.

\section{Notions and preliminaries}

Throughout this paper we will assume that $X$ is a Banach space and $\mathcal{B}(X)$ is the Banach algebra of all linear and bounded operators acting on $X$.
We denote by
$$T=\{(t,s)\in\mathbb{R}_{+}^{2} : t\geq{s}\}.$$

Let us first recall the classical notion of evolution operator:
\begin{definition}\label{d1}
An application  $U:T\longrightarrow{\mathcal{B}} \left(X\right)$ is called
{\it an evolution operator} if it has the following properties:

$e_1)$ $U(t,t)=I$ (the identity on $X$ ), for every  $t\geq 0$;

$e_2)$ $U(t,s)U(s,t_{0})=U(t,t_{0})$, for all $(t,s),\,(s,t_0)\in T$;

$e_3)$ for each $x\in X$ the function
\[
T\ni(t,s)\longmapsto U(t,s)x\in X
\]
is continuous.
\end{definition}

\begin{definition}\label{def2.2.1}
An evolution operator $U:T\longrightarrow \mathcal{B}(X)$ is said to be
{\it uniformly exponentially unstable} if there are $ N\geq 1$ and $\nu > 0$ such that
\begin{equation}\label{eq.u.e.i}
N \parallel U(t, t_{0} ) x_0 \parallel \,\geq \,e^{\nu(t-t_0)}\parallel{x_0}\parallel,\text{ for all  $(t, t_{0},x_0)\in T\times X$.}
\end{equation}

\end{definition}

\begin{remark}\rm
The evolution operator $U$ is uniformly exponentially unstable if and only if there are $ N\geq 1$ and $\nu > 0$ such that
$$ N \parallel U(t, t_{0} ) x_0 \parallel \,\geq \,e^{\nu(t-s)}\parallel U(s,t_0){x_0}\parallel,$$
for all $(t,s),(s, t_{0})\in T $ and $x_0\in X.$

This means that the orbit $U(\cdot,t_0)x_0$ is uniformly exponentially unstable, for all initial data $(t_0,x_0)\in\mathbb{R}_{+}\times X$.
\end{remark}

The concept of uniform exponential instability was generalized in a nonuniform way in \cite{Me.Sa.Sa.2001-1}, as follows:

\begin{definition}\label{def2.2.1-1}
An evolution operator $U:T\longrightarrow \mathcal{B}(X)$ is said to be
{\it non-uniformly exponentially unstable} if there are $N:\mathbb{R}_{+}\longrightarrow[1,\infty)$ and $\nu>0$ such that
\begin{equation}\label{eq.e.i}
N(t) \parallel U(t, t_{0} ) x_0 \parallel \,\geq \,e^{\nu(t-t_0)}\parallel{x_0}\parallel,\text{ for all  $(t, t_{0},x_0)\in T\times X$.}
\end{equation}

\end{definition}

A particular case of nonuniform exponential instability was considered by L. Barreira and C. Valls in \cite{Ba.Va.2009-2} as a part of nonuniform exponential dichotomy:

\begin{definition}\label{def2.2.1-2}
The evolution operator $U:T\longrightarrow \mathcal{B}(X)$ is said to be
{\it exponentially unstable} if there are $ N\geq 1$, $\alpha\geq 0$ and $\nu > 0$ such that
\begin{equation}\label{eq.n.e.i}
Ne^{\alpha t} \parallel U(t, t_{0} ) x_0 \parallel \geq e^{\nu(t-t_0)}\parallel{x_0}\parallel,\text{ for all  $(t, t_{0},x_0)\in T\times X$.}
\end{equation}

\end{definition}

\begin{remark}\rm
If $U$ is non-uniformly exponentially unstable, does not necessarily result that  $U$ is exponentially unstable.
\end{remark}

\begin{example}\rm
Let  $u:\mathbb{R}_{+}\longrightarrow [1,\infty)$ be a continuous function with the properties
$u\left(n+\frac{1}{n}\right)=e^{n^2}$ and $u(n)=1$, for all $n\in \mathbb{N}^{*}$.
The evolution operator $U$ defined by
$$U(t,s)=\frac{u(s)}{u(t)}e^{t-s}I,\,\forall\,(t,s)\in T$$
is non-uniformly exponentially unstable and it is not exponentially unstable.
\end{example}
\begin{proof}
Obviously, we have that
\begin{equation*}
u(t)\parallel U(t,t_0)x_0\parallel\geq e^{t-t_0}\parallel x_0\parallel,
\end{equation*}
for all  $(t,t_0,x_0)\in T\times X$. Therefore, $U$ is non-uniformly exponentially unstable. If we suppose that $U$ is exponentially unstable, there are $N\geq 1$, $\alpha\geq 0$ and $\nu>0$ such that
\begin{equation*}
Ne^{\alpha t}\parallel U(t,t_0)x_0\parallel \geq e^{\nu(t-t_0)}\parallel x_0\parallel,\text{ for all  $(t,t_0,x_0)\in T\times X.$}
\end{equation*}

In particular, for $t=n+\frac{1}{n}$, $t_0=n$, $n\in \mathbb{N}^{*}$ and $x_0\in X$ with $\parallel x_0\parallel=1$ we obtain
$$Ne^{\alpha(n+\frac{1}{n})}e^{\frac{1}{n}}\geq e^{\frac{\nu}{n}}e^{n^2}, \,\forall\,n\in \mathbb{N}^{*},$$ which is false. Therefore, $U$ is not exponentially unstable.

\end{proof}

One of the most important properties that defines the concept of uniform exponential instability is
\begin{equation}\label{limita}
\lim\limits_{t\rightarrow\infty}\parallel U(t,t_0)x_0\parallel=\infty,\text{ for } (t_0,x_0)\in\mathbb{R}_{+}\times X, \;x_0\neq0.
\end{equation}
For the concepts of exponential instability considered in Definition \ref{def2.2.1} and Definition \ref{def2.2.1-1}, this property is not necessarily fulfilled. For example, the evolution operator
$$U(t,s)=e^{-(t-s)}I,\text{ for } (t,s)\in T$$
is (non-uniformly) exponentially unstable and
$$\lim\limits_{t\rightarrow\infty}\parallel U(t,t_0)x_0\parallel=0,\text{ for every } (t_0,x_0)\in\mathbb{R}_{+}\times X.$$

If we suppose in Definition \ref{def2.2.1-1} that $\alpha\in[0,\nu)$, the condition (\ref{limita}) still holds true.
In this paper, we propose another type of  exponential instability for which condition (\ref{limita}) remains valid.
\begin{definition}\label{def2.2.2}
An evolution operator  $U$ is said to be {\it weakly exponentially unstable} if there are constants $N\geq1$ and $\nu>0$ such that for each  $x_0\in X$ there is $t_0=t_0(x_0)\geq 0$ with the property
\begin{equation}\label{eq.w.e.i}
N\parallel U(t,t_0)x_0 \parallel\,\geq\,e^{\nu(t-s)}\parallel U(s,t_0)x_0 \parallel,\text{ for all  $t\geq s\geq t_{0}$.}
\end{equation}

\end{definition}

We remark that this concept of exponential instability is more general than the individual instability (the corespondent for the individual stability from \cite{Bu.Po.2001}  and \cite{Pr.Po.Pr.2006-1}), where $(t_0,x_0)\in\mathbb{R}_{+}\times X$ is fixed.

\begin{remark}\rm
If the evolution operator $U$  is uniformly exponentially unstable
then it is weakly exponentially unstable. The following example shows that, even in finite dimensional Hilbert spaces, the converse is not valid.
\end{remark}

\begin{example}\rm
Let $X=\mathbb{R}^{2}$ with the Euclidian norm. The
evolution operator
$$U(t,t_0)(x_1,x_2)=(\xi_1,\xi_2),$$
where
$$\xi_1=e^{t-t_0}\cos{t}\left(x_1\cos{t_0}+x_2\sin{t_0}\right)+e^{-(t-t_0)}\sin{t}\left(x_1\sin{t_0}-x_2\cos{t_0}\right)$$
$$\xi_2=e^{t-t_0}\sin{t}\left(x_1\cos{t_0}+x_2\sin{t_0}\right)-e^{-(t-t_0)}\cos{t}\left(x_1\sin{t_0}-x_2\cos{t_0}\right)$$
is weakly exponentially unstable, but it is not uniformly exponentially unstable.
\end{example}
\begin{proof}
For each $x_0 \in \mathbb{R}^{2}$ there are  $r_0\geq 0$ and $t_0\in[0,2\pi)$ such that
$$x_0=(r_0 \cos{t_0},r_0 \sin{t_0}).$$
A simple computation shows that
\begin{equation*}
U(t,t_0)x_0=(r_0 e^{\,t-t_0}\cos{t},r_0 e^{\,t-t_0}\sin{t})
\end{equation*}
and so
$$\parallel U(t,t_0)x_0\parallel=r_0 e^{\,t-t_0}=e^{\,t-s}\parallel U(s,t_0)x_0\parallel, \text{ for all $t\geq s\geq t_0$.}$$
This implies that $U$ is weakly exponentially unstable.
On the other hand, for $x_0=(-\sin{t_0}, \cos{t_0})$ we obtain
$$ U(t,t_0)x_0=(-e^{-(t-t_0)}\sin{t}, e^{-(t-t_0)} \cos{t})$$
and hence
$$\parallel U(t,t_0)x_0\parallel=e^{-(t-t_0)}=e^{-(t-s)}\parallel U(s,t_0)x_0\parallel,$$
which proves that $U$ is not uniformly exponentially unstable.
\end{proof}

The next result may be regarded as an equivalent definition for the concept of weak exponential instability.

\begin{proposition}\label{p.f.i}
An evolution operator  $U:T\longrightarrow \mathcal{B}(X)$ is weakly exponentially unstable  iff there is a nondecreasing function
$f:\mathbb{R}_{+}\longrightarrow (0,\infty)$ with $\lim\limits_{t\rightarrow \infty} f(t)=+\infty$ such that for each $x_0\in X$ there is $t_0\geq 0$ with the property
\begin{equation}\label{eq.f.i}
\parallel U(t,t_0)x_0 \parallel\,\geq f(t-s)\parallel U(s,t_0)x_0\parallel,\text{ for all  $t\geq s\geq t_0$.}
\end{equation}

\end{proposition}

\begin{proof} \emph{Necessity} is a simple computation for  $f(t)=\frac{1}{N} e^{\nu t}$,  $\forall\,t\geq 0$.

\emph{Sufficiency}. From $\lim\limits_{t\rightarrow \infty} f(t)=+\infty$ it results that there is $c>0$ such that $f(c)>1$. Letting $\nu>0$  such that $f(c)=e^{\nu c}$, we consider $N=\frac{f(c)}{f(0)}\geq 1$. From hypothesis we have that for each  $x_0\in X$ there is $t_0\geq 0$ satisfying relation (\ref{eq.f.i}).

For all $t\geq s\geq t_0$ there are $n\in \mathbb{N}$ and $r\in [0,c)$ such that $t=s+nc+r$. Successively, we obtain
\begin{align*}
f(c)\parallel U(t,t_0)x_0\parallel &\geq f(c)f(r)\parallel U(s+nc,t_0)x_0\parallel\\
&\geq f(0)f(c)\parallel U(s+nc,t_0)x_0\parallel\\
&\geq \cdots \geq f(0)f(c)^{n+1}\parallel U(s,t_0)x_0\parallel\\
&\geq f(0)e^{\nu(t-s)}\parallel U(s,t_0)x_0\parallel,
\end{align*}
for all $t\geq s\geq t_0$.
Thus, $$N\parallel U(t,t_0)x_0\parallel\geq e^{\nu(t-s)}\parallel U(s,t_0)x_0\parallel,$$
for all $t\geq s\geq t_0$, which proves that $U$ is weakly exponentially unstable.

\end{proof}

\section{The main results}

An important characterization for weak exponential instability is given by:

\begin{theorem}\label{T.Datko.i-1}
The evolution operator $U:T\longrightarrow \mathcal{B}(X)$ is weakly exponentially unstable if and only if there are   $p,M,K\geq 1$ and $\omega>0$ such that for each $x_0\in X$ there is $t_0\geq 0$ with
\begin{enumerate}
\item[(i)]
$\parallel U(s,t_0)x_0\parallel\leq M e^{\omega(t-s)} \parallel U(t,t_0)x_0\parallel,$ for all $t\geq s\geq t_0;$
\item[(ii)]
$\int\limits_{t_0}^{t}\parallel U(\tau,t_0)x_0\parallel^{p}d\tau\leq K\parallel U(t,t_0)x_0\parallel^{p},$
for all  $t\geq t_0$.

\end{enumerate}
\end{theorem}

\begin{proof}
\emph{Necessity}. Obviously, if $U$ is weakly exponentially unstable, condition (i) holds. For $p\geq 1$ and $K=\max\left\{\frac{N^{p}}{\nu p},1\right\}$, where $N\geq 1$ and $\nu>0$ are given by Definition \ref{def2.2.2}, we have that for each $x_0\in X$ there is $t_0\geq 0$ such that
\begin{align*}
\int\limits_{t_0}^{t}\parallel U(\tau,t_0)x_0\parallel^{p}d\tau&\leq N^{p}\int\limits_{t_0}^{t}e^{-\nu p(t-\tau)}d\tau \parallel U(t,t_0)x_0\parallel^{p}\\
&\leq K \parallel U(t,t_0)x_0\parallel^{p},\text{ for all $t\geq t_0$.}
\end{align*}

\emph{Sufficiency}. We suppose that there are  $p,M,K\geq 1$ and $\omega>0$ such that for each $x_0\in X$ there is $t_0\geq 0$ satisfying relations (i) and (ii).

For  $t\geq s+1>s\geq t_0$ we have
\begin{align*}
M^{p}K\parallel U(t,t_0)x_0\parallel^{p}&\geq M^{p}\int\limits_{t_0}^{t}\parallel U(\tau,t_0)x_0\parallel^{p}d\tau\\
&\geq\int\limits_{s}^{t}e^{-\omega p(\tau-s)}d\tau\parallel U(s,t_0)x_0\parallel^{p}\\
&\geq \int\limits_{0}^{1}e^{-\omega p u}du \parallel U(s,t_0)x_0\parallel^{p}\\
&=\frac{1-e^{-\omega p}}{\omega p} \parallel U(s,t_0)x_0\parallel^{p}
\end{align*}
and for $t\in [s,s+1)$ it follows that
$$M^{p}\parallel U(t,t_0)x_0\parallel^{p}\geq e^{-\omega p}\parallel U(s,t_0)x_0\parallel^{p}.$$
Hence
\begin{equation*}
\parallel U(t,t_0)x_0\parallel^{p}\geq L\parallel U(s,t_0)x_0\parallel^{p},\text{ for all } t\geq s\geq t_0,
\end{equation*}
where $L=\min\left\{\frac{1}{M^{p}K}\frac{1-e^{-\omega p}}{\omega p},\frac{e^{-\omega p}}{M^{p}}\right\}.$

On the other hand
\begin{align*}
K\parallel U(t,t_0)x_0\parallel^{p}&\geq \int\limits_{t_0}^{t}\parallel U(\tau,t_0)x_0\parallel^{p}d\tau\geq\int\limits_{s}^{t}\parallel U(\tau,t_0)x_0\parallel^{p}d\tau\\
&\geq L(t-s)\parallel U(s,t_0)x_0\parallel^{p},\text{ for all $t\geq s\geq t_0$.}
\end{align*}

This proves that
\begin{equation*}
\left(1+K^{1/p}\right) \parallel U(t,t_0)x_0\parallel \geq L^{1/p}\left[1+(t-s)^{1/p}\right]\parallel U(s,t_0)x_0\parallel,
\end{equation*}
for all $t\geq s\geq t_0$. By Proposition \ref{p.f.i}, we conclude that $U$ is weakly exponentially unstable.

\end{proof}

\begin{remark}\rm
Theorem \ref{T.Datko.i-1} can be considered a version for the case of weak exponential instability of a well-known result due to Datko \cite{Da.1973}.
\end{remark}

In the following corollary, we give a discrete version of the previous theorem.

\begin{corollary}
The evolution operator $U:T\longrightarrow \mathcal{B}(X)$ is weakly exponentially unstable if and only if there are   $p,M,K\geq 1$ and $\omega>0$ such that for each $x_0\in X$ there is $t_0\geq 0$ with
\begin{enumerate}
\item[(i)]
$\parallel U(s,t_0)x_0\parallel\leq M e^{\omega(t-s)} \parallel U(t,t_0)x_0\parallel,$ for all $t\geq s\geq t_0;$
\item[(ii)]
$\sum\limits_{n=0}^{[t-t_0]}\parallel U(t-n,t_0)x_0\parallel^{p} \leq K\parallel U(t,t_0)x_0\parallel^{p},$
for all  $t\geq t_0$.
\end{enumerate}
\end{corollary}

\begin{proof}
\emph{Necessity}. By Definition \ref{def2.2.2}, we have that there are $N\geq 1$ and $\nu >0$ with the property that for each $x_0\in X$ there is $t_0\geq 0$ such that
\begin{align*}
\sum\limits_{n=0}^{[t-t_0]}\parallel U(t-n,t_0)x_0\parallel^{p}&\leq N^{p}\left(\sum\limits_{n=0}^{[t-t_0]} e^{-\nu p n}\right)\parallel U(t,t_0)x_0\parallel^{p}\\
&\leq K\parallel U(t,t_0)x_0\parallel^{p},
\end{align*}
for all $t\geq t_0$, where
$p\geq 1$ is fixed and $K=\max\left\{\frac{N^{p}}{1-e^{-\nu p}},1\right\}.$

\emph{Sufficiency}. Let $p\geq 1$ and $K>0$ with the property that for each $x_0\in X$ there is $t_0\geq 0$ such that (i) and (ii) hold. This implies
\begin{align*}
\int\limits_{t_0}^{t} \parallel U(\tau,t_0)x_0\parallel^{p}d\tau&=\int\limits_{0}^{t-t_0} \parallel U(t-s,t_0)x_0\parallel^{p}ds\\
&\leq \sum\limits_{n=0}^{[t-t_0]}\int\limits_{n}^{n+1}\parallel U(t-s,t_0)x_0\parallel^{p}ds\\
&\leq M^{p} \sum\limits_{n=0}^{[t-t_0]}\int\limits_{n}^{n+1}e^{\omega p(s-n)}ds \parallel U(t-n,t_0)x_0\parallel^{p}\\
&\leq M^{p}e^{\omega p} \sum\limits_{n=0}^{[t-t_0]} \parallel U(t-n,t_0)x_0\parallel^{p}\\
&\leq KM^{p}e^{\omega p}\parallel U(t,t_0)x_0\parallel^{p},
\end{align*}
for all  $t\geq t_0$.

Applying  Theorem \ref{T.Datko.i-1} we deduce that $U$ is weakly exponentially unstable.
\end{proof}

\begin{definition}\label{def.f.Ly}\rm
A function $L:T\times X\longrightarrow \mathbb{R}$ is said to be a \emph{Lyapunov function} for the evolution operator $U$ if for each $x_0\in X$ there is $t_0\geq 0$  such that  $L$ is a solution of the Lyapunov equation
\begin{equation}\label{eq.Ly.}
L(t,t_0,x_0)+\int\limits_{s}^{t} \parallel U(\tau,t_0)x_0\parallel ^2\,d\tau= L(s,t_0,x_0),\;t\geq s\geq t_0.
\end{equation}
\end{definition}

In the following, we give a characterization for weak exponential instability in terms of the existence of a Lyapunov function.

\begin{theorem}
The evolution operator $U:T\longrightarrow \mathcal{B}(X)$ is weakly exponentially unstable if and only there are a Lyapunov function $L:T\times X\longrightarrow \mathbb{R}_{-} $  and  $M,\omega,m>0$ such that
\begin{enumerate}
\item[(i)]
$\parallel U(s,t_0)x_0\parallel\leq M e^{\omega(t-s)} \parallel U(t,t_0)x_0\parallel,$ for all $t\geq s\geq t_0;$
\item[(ii)]
$ | L(t,t_0,x_0) |\leq m\parallel U(t,t_0)x_0\parallel^2,\text{ for all } t\geq t_0.$
\end{enumerate}
\end{theorem}

\begin{proof}
\emph{Necessity}. We define the function
$$L(t,s,x)=-\int\limits_{s}^{t}\parallel U(\tau,s)x\parallel^2d\tau, \text{ for } (t,s)\in T \text{ and } x\in X.$$

It is obvious that $L(t,s,x)\leq 0$ and from Definition \ref{def2.2.2} we have that there are $N\geq 1$ and $\nu>0$ such that for each $x_0\in X$ there is $t_0\geq 0$ with
\begin{equation*}
|L(t,t_0,x_0)|=\int\limits_{t_0}^{t}\parallel U(\tau,t_0)x_0\parallel^{2}d\tau\leq m\parallel U(t,t_0)x_0\parallel^{2},
\end{equation*}
for all $t\geq t_0$, where $m=\frac{N^{2}}{2\nu}>0$.
A simple computation shows
$$L(s,t_0,x_0)-L(t,t_0,x_0)-\int\limits_{s}^{t}\parallel U(\tau,t_0)x_0\parallel^{2}d\tau=0, \text{ for all $t\geq s\geq t_0$.}$$

\emph{Sufficiency}. We suppose that there are a Lyapunov function $L:T\times X\longrightarrow \mathbb{R}_{-} $  and $M,\omega,m>0$ such that (i) and (ii) hold.

Then
\begin{align*}
\int\limits_{t_0}^{t}\parallel U(\tau,t_0)x_0\parallel^{2}d\tau&= L(t_0,t_0,x_0)-L(t,t_0,x_0)\leq-L(t,t_0,x_0)\\
&=|L(t,t_0,x_0)|\leq m\parallel U(t,t_0)x_0\parallel^2,
\end{align*}
for all $t\ge t_0$. By Theorem \ref{T.Datko.i-1} we deduce that $U$ is weakly exponentially unstable.

\end{proof}

\begin{remark}\rm
The preceding theorem is an extension for the case of weak exponential instability of a result due to Megan and Bu\c se in \cite{Me.Bu.1992}.
\end{remark}

\section*{Acknowledgement}

The work is supported by CNCSIS -- UEFISCSU, project number PN II - IDEI  1080/2008.

\end{document}